\documentclass[a4paper,10pt]{article}

\usepackage{enumerate}
\usepackage{amsmath}
\usepackage{amssymb}
\usepackage{amsthm}
\usepackage[all]{xy}
\usepackage{array}                              
\usepackage{setspace}
\usepackage{hyperref}
\hypersetup{
    colorlinks,
    citecolor=black,
    filecolor=black,
    linkcolor=black,
    urlcolor=black
}

\newtheorem{theorem}{Theorem}[section]
\newtheorem{corollary}[theorem]{Corollary}
\newtheorem{lemma}[theorem]{Lemma}
\newtheorem{proposition}[theorem]{Proposition}

\theoremstyle{remark}

\theoremstyle{definition}
\newtheorem{definition}[theorem]{Definition}

\newcommand{\N}{\mathbb{N}}

\newcommand{\F}{\mathbb{F}}
\newcommand{\E}{\mathcal{E}}

\newcommand{\End}{\mathrm{End}}

\newcommand{\Hom}{\mathrm{Hom}}

\newcommand{\IBr}{\mathrm{IBr}}

\newcommand{\Ind}{\mathrm{Ind}}

\newenvironment{mathex}[1][LLLLLLLLLLLL]{
\[%
\newcolumntype{L}{>{\displaystyle\setlength{\arraycolsep}{4pt}}l}%
\newcolumntype{C}{>{\displaystyle\setlength{\arraycolsep}{4pt}}c}%
\newcolumntype{R}{>{\displaystyle\setlength{\arraycolsep}{4pt}}r}%
\setlength{\arraycolsep}{1.5pt}%
\begin{array}{>{\vspace*{1.3ex}}#1}%
}{
\end{array}%
\vspace*{-1.3ex}%
\]%
\ignorespacesafterend%
}

\title{On Imprimitive Representations of Finite Reductive Groups in Non-defining Characteristic}
\date{}

\author{Matthias Klupsch}

\begin{document}
 \maketitle
 \vspace{-1cm}
 \begin{center}
  \small{RWTH Aachen University, Lehrstuhl D f\"{u}r Mathematik, \\ Pontdriesch 14/16, 52062 Aachen, Germany}
 \end{center}

\begin{abstract}
 In this paper, we begin with the classification of Harish-Chandra imprimitive representations
 in non-defining characteristic. We recall the connection of this problem to certain generalizations of 
 Iwahori-Hecke algebras and show that Harish-Chandra induction is compatible with the Morita equivalence 
 by Bonnaf\'{e} and Rouquier, thus reducing the classification problem to quasi-isolated blocks.
 Afterwards, we consider imprimitivity of unipotent representations of certain classical groups. In the case 
 of general linear and unitary groups, our reduction
 methods then lead to results for arbitrary Lusztig series.

 \textbf{Keywords}:
 modular representation theory, finite reductive groups,
 Harish-Chandra induction, imprimitive representation 
 
 \textbf{Mathematics Subject Classification 2010}:  20C33
 \end{abstract}
 \section{Introduction}
 In \cite{hiss2015imprimitive} and \cite{hiss2017imprimitive}, the imprimitive representations 
 of finite quasi-simple groups in characteristic $0$ were classified and some results 
 were obtained for arbitrary characteristic.
 Focusing now on positive characteristic, one big part of the classification problem of imprimitive
 representations revolves around Harish-Chandra imprimitive representations of finite reductive groups.
    
 Let $G$ be a connected reductive group over the algebraic closure $\F$ of a field $\F_q$ with $q$ elements and suppose that $G$ is 
 defined over $\F_q$ via a Frobenius morphism $F : G \to G$.
 We consider a prime $\ell$ not dividing $q$ and an $\ell$-modular system $(K,R,k)$ 
 which is split for $G^F$ and all its subgroups.
 
 For an $F$-stable parabolic subgroup $P \subseteq G$ with unipotent radical $V$ and $F$-stable Levi complement $L$ and $\Lambda \in \{K,R,k\}$ we consider 
 Harish-Chandra induction 
 $$ R^G_{L \subseteq P} : \Lambda L^F \mathrm{-mod} \to  \Lambda G^F \mathrm{-mod}, X \mapsto \Lambda \: G^F/V^F \otimes_{\Lambda L^F} X$$
 and Harish-Chandra restriction 
 $$ {}^*\!R^G_{L \subseteq P} : \Lambda G^F \mathrm{-mod} \to  \Lambda L^F \mathrm{-mod}, X \mapsto \Lambda \: V^F \backslash G^F \otimes_{\Lambda G^F} X.$$
 These two functors are left and right adjoints of one another and fixing the split Levi subgroup $L$ while changing the parabolic subgroup gives rise to  naturally equivalent functors.
 
 An imprimitive representation is one which is induced from a proper subgroup.
 When only considering Harish-Chandra induced representations, we get the notion of Harish-Chandra imprimitivity.
 \begin{definition}
  A simple $k G^F$-module $X$ is called (Harish-Chandra) imprimitive if $R^G_L X' \cong X$ for some $k L^F$-module $X'$
  where $L \subseteq G$ is a proper split Levi subgroup of $G$.
  If $X$ is not imprimitive, it is said to be primitive.
 \end{definition}
 Since we are only considering the concept of Harish-Chandra imprimitivity in the following, there will be no confusion in calling it simply imprimitivity.
 It is also noteworthy that by \cite[Prop. 7.1]{hiss2015imprimitive}, the notion of Harish-Chandra imprimitivity coincides with the 
 more general notion of imprimitivity for quasi-simple groups of Lie type.
 
 It takes multiple steps to reach a classification of Harish-Chandra imprimitive representations in non-defining characteristic.
It makes sense to look at unipotent representations of groups with connected center first
with the idea of reducing the general problem to this case or similar cases.
For most of the classical groups, one can apply a result by Christoph Sch\"{o}nnenbeck 
on Iwahori-Hecke algebras (\cite{schoennenbeck2017induced}) to the endomorphism algebras of the Harish-Chandra induction of 
cuspidal modules to find that imprimitivity is quite rare for unipotent representations but actually 
does occur in contrast to the analogous situation in characteristic $0$ (\cite[Corollary 8.5]{hiss2015imprimitive}).

To get from unipotent representations to arbitrary ones is more difficult 
in positive characteristic than in characteristic $0$. For instance, we do not have an analogue to  
the Jordan decomposition of characters.
However, we can at least make use of the Morita equivalence by Bonnaf\'{e} and Rouquier from \cite{bonnafe2003categories} 
to reduce our problem to the study of representations in (quasi-)isolated Lusztig series.
For this, we shall prove that Harish-Chandra induction commutes with this kind of Morita 
equivalences.
As a consequence, we will be able to extend our results on unipotent representations to arbitrary Lusztig series 
for the general linear and unitary groups.

 \section{Imprimitivity and Hecke Algebras}
 In characteristic $0$, the property of a representation of $G^F$ to be imprimitive turns out to be 
 a property of the Harish-Chandra series it belongs to. In fact \cite[Thm. 8.3]{hiss2015imprimitive} implies 
 that if there is one imprimitive representation in a Harish-Chandra series, then all 
 representations in this series are imprimitive.
 The proof of this theorem relies heavily 
 on the knowledge of the algebras $\End_{K G^F}(R^G_{L_0} X_0)$ 
 where $X_0$ is a cuspidal $K L_0^F$-module.
 In this section, we shall review the properties of corresponding algebras in characteristic $\ell$.

 Let $(L_0,X_0)$ be a cuspidal pair of $(G,F)$, that is, $L_0 \subseteq G$
 is a split Levi subgroup of $G$ and $X_0$ is a simple cuspidal $k L_0^F$-module.
 We consider the full subcategory $k G^F\mathrm{-mod}_{(L_0,X_0)}$ of $k G^F\mathrm{-mod}$ 
 whoses objects $X$ admit a monomorphism $X \to (R^G_{L_0} X_0)^n $ and an epimorphism 
 $(R^G_{L_0}X_0)^m \to X$ for some positive integers $m,n \in \N$.
 In particular, every simple $k G^F$-module belonging to the Harish-Chandra series of $(L_0,X_0)$ 
 is an object in $k G^F\mathrm{-mod}_{(L_0,X_0)}$ by \cite[Thm. 1.28]{cabanes2004representation}.
 As an important consequence of this, we note that the simple objects in $k G^F\mathrm{-mod}_{(L_0,X_0)}$, that is, the non-zero objects whose only proper subobjects are zero-objects, 
 are precisely the simple $k G^F$-modules belonging to the Harish-Chandra series of $(L_0,X_0)$.
 
 The functor 
 $$ \Hom_{k G^F}(R^G_{L_0}X_0 , - ): k G^F\mathrm{-mod}_{(L_0,X_0)} \to \End_{k G^F}(R^G_{L_0} X_0)^\circ\mathrm{-mod}  $$
 is an equivalence of categories by \cite[Thms. 1.20, 1.25]{cabanes2004representation}.
 
 Recall that the algebra $\End_{k G^F}(R^G_{L_0} X_0)^\circ$ has a $k$-basis $\{B_w\}_{w}$ indexed by $w \in N_{G^F}(L_0,X_0)/{L_0}^F$,
 where $N_{G^F}(L_0,X_0) = \{g \in N_{G}(L_0)^F \:|\: {}^g X_0 \cong X_0\}$.
 This algebra is akin to Iwahori-Hecke algebras in view of \cite[Thm. 3.12]{geck1996towards}.

 \begin{lemma}\label{la: HC commutative diagram}
 Given any split Levi subgroup $L \subseteq G$ with $L_0 \subseteq L$,
 the algebra morphism $R^G_{L} : \End_{k L^F}(R^{L}_{L_0}X_0) \to \End_{k G^F}(R^{G}_{L_0}X_0)$ 
 is injective with $R^G_{L}(B_w) = B_w$ for all $w \in N_{L^F}(L_0,X_0)/L_0^F$.
 Moreover, denoting by $\Ind_{R^G_L}$ the induction functor associated with this morphism  
 the diagram 
 \[
  \xymatrix{
   k G^F\mathrm{-mod}_{(L_0,X_0)} \ar[rrrr]^{\Hom_{k G^F}(R^G_{L_0}X_0 , - )} &&&& \End_{k G^F}(R^G_{L_0} X_0 )^\circ\mathrm{-mod} \\ \\     
   k L^F\mathrm{-mod}_{(L_0,X_0)} \ar[rrrr]^{\Hom_{k L^F}(R^L_{L_0}X_0 , - )} \ar[uu]^{R^G_L}&&&& \End_{k L^F}(R^L_{L_0} X_0)^\circ\mathrm{-mod} \ar[uu]^{\Ind_{R^G_{L}}}
  }
 \]
 is commutative up to natural isomorphism.
 \end{lemma}
 \begin{proof}
  Let us set $X = R^L_{L_0} X_0$. Since Harish-Chandra induction is faithful, the morphism 
  $R^G_L : \Hom_{k L^F}(X,X) \to \Hom_{k G^F}(R^G_L X,R^G_L X)$ is injective.
  The identity $R^G_L(B_w) = B_w$ for all $w \in N_{L^F}(L_0,X_0)/L_0^F$ follows with a simple calculation from the 
  definition of the $B_w$ \cite[(3.5)]{geck1996towards} and the transitivity of Harish-Chandra induction.
   
  For $k L^F$-modules $Y$ and $Z$, we consider the natural map  
  \begin{mathex}[RLL]
  \Hom_{k L^F}(Z, X)  \otimes_{\End_{k L^F}(X)^\circ}  \Hom_{k L^F}(X, Y)  &\to& \Hom_{k L^F}(Z, Y), \\
  \varphi \otimes f &\mapsto& f \circ \varphi.
  \end{mathex}
  This is an isomorphism for $Z = X$ and thus also for $Z$ being a direct sum of copies of $X$.
  In particular, by the Mackey formula \cite[Thm. 5.1]{digne1991representations}, 
  \begin{mathex}[RLL]
  \Hom_{k L^F}({}^*\!R^{G}_L R^{G}_L X, X)  \otimes_{\End_{k L^F}(X)^\circ}  \Hom_{k L^F}(X, Y)  &\to& \Hom_{k G^F}({}^*\!R^{G}_L R^{G}_L X, Y),  \\
     \varphi \otimes f &\mapsto& f \circ \varphi
  \end{mathex}
  is an isomorphism natural in $Y \in  k L^F\mathrm{-mod}_{(L_0,X_0)}$ and by adjointness, the map 
  \begin{mathex}[RLL]
  \End_{k G^F}(R^G_L X)^\circ  \otimes_{\End_{k L^F}(X)^\circ}  \Hom_{k L^F}(X, Y)  &\to& \Hom_{k G^F}(R^G_L X,R^G_L Y), \\
   \varphi \otimes f &\mapsto& R^G_L(f) \circ \varphi
  \end{mathex}
  is an isomorphism, too.
 \end{proof}
  This result tells us that finding the imprimitive representations in the Harish-Chandra series of $(L_0, X_0)$ amounts 
  to finding the simple $\End_{k G^F}(R^G_{L_0} X_0)^\circ$-modules 
  which are of the form $\Ind_{R^G_L}(M)$ for some simple $\End_{k L^F}(R^{L}_{L_0} X_0)^\circ$-module $M$.
  The easiest case is the following.
  \begin{corollary}\label{cor: normalizers equal is sufficient}
   Let $L \subseteq G$ be a proper split Levi subgroup of $G$ containing $L_0$.
   If $N_{G^F}(L_0,X_0) = N_{L^F}(L_0,X_0)$, then every simple $k G^F$-module
   belonging to the Harish-Chandra series of $(L_0,X_0)$ is imprimitive.
  \end{corollary}
  \begin{proof}
   If $N_{G^F}(L_0,X_0) = N_{L^F}(L_0,X_0)$, 
   then $$R^G_{L} : \End_{k L^F}(R^{L}_{L_0}X_0) \to \End_{k G^F}(R^{G}_{L_0}X_0)$$
   is a monomorphism between $k$-vector spaces of the same dimension and thus an isomorphism. 
   By Lemma (\ref{la: HC commutative diagram}), this implies that  
   every simple $k G^F$-module
   belonging to the Harish-Chandra series of $(L_0,X_0)$ is Harish-Chandra
   induced.
  \end{proof}
  It is conjectured that the converse of this corollary also holds true if the center of $G$ is connected as well as that  
  imprimitivity of a $k G^F$-module implies   
  the imprimitivity of every other module in the same Harish-Chandra series as it is the case in 
  in characteristic $0$.
 \section{The Bonnaf\'{e} -- Rouquier Morita equivalence}
 
In characteristic $0$, as was mentioned before, imprimitivity can be viewed as a property of 
Harish-Chandra series. Moreover, it was proven in \cite[Thm. 7.3, Thm 8.4]{hiss2015imprimitive} that 
imprimitivity can also be viewed as a property of Lusztig series in characteristic $0$. 

Lusztig series are compatible with modular representation theory as was shown by Brou\'{e} and Michel in \cite{broue1989blocs}.
In particular, certain unions of Lusztig series turn out to be unions of $\ell$-blocks.

In \cite{bonnafe2003categories}, Bonnaf\'{e} and Rouquier showed that every $\ell$-block 
of a finite reductive group is Morita equivalent
to some quasi-isolated $\ell$-block of a possibly different finite reductive group. 
In this section, we shall show that this Morita equivalence is compatible with Harish-Chandra induction 
which also implies that it preserves and reflects the property of being imprimitive.
To do so we shall need a result by Bonnaf\'{e}, Dat and Rouquier from \cite{bonnafe2015derived} which gives a sufficient condition for 
Lusztig induction to depend only on the Levi subgroup (and not on the parabolic subgroup).

So let $(G^*,F^*)$ be a group in duality with $(G,F)$.
Recall that we have a decomposition
$$ \Lambda G^F = \bigoplus_{[s]} \Lambda G^F e_s^{G^F}$$
into sums of blocks corresponding to the decomposition 
$$ \IBr(G^F) = \bigcup_{[s]} \E_{\ell}(G,F,[s])$$
into $\ell$-modular Lusztig series. Both decompositions are indexed by the conjugacy classes 
of semisimple $\ell'$-elements in ${G^*}^{F^*}$.

Let us fix a semisimple element $s \in (G^*)^{F^*}_{\ell'}$ and let $G_s^* \subseteq G^*$
be a rational Levi subgroup containing $C_{G^*}(s)$. Let $G_s \subseteq G$ correspond 
to $G_s$ under duality. 
If $P_s = G_s V_s$ denotes a parabolic subgroup of $G$ with Levi complement $G_s$,
then the Deligne-Lusztig variety 
$$ Y^{G}_{G_s \subseteq P_s} = \{g V_s \in G/V_s \:|\: g^{-1}F(g) \in V_s F(V_s)\} $$
has the property $H_c^i(Y^{G}_{G_s \subseteq P_s}, \Lambda) = 0$ except for $i = d := \dim(Y^{G}_{G_s \subseteq P_s})$, 
and $H_c^d(Y^{G}_{G_s \subseteq P_s}, \Lambda)$ induces a Morita equivalence between 
the sum of blocks $\Lambda G_s^F e_s^{G_s^F}$ and $\Lambda G e_s^{G^F}$ where $e_s^{G_s^F}$ and $e_s^{G^F}$
denote the central idempotents corresponding to the Lusztig series $\E_\ell(G_s, F, [s])$ and $\E_\ell(G,F,[s])$,
respectively. This is the main result of \cite{bonnafe2003categories}.

Now, let $L \subseteq G$ be a split Levi subgroup, $L^* \subseteq G^*$ a dual correspondent and 
suppose that $s \in L^*$.
We set $L_s = L \cap G_s$ and $L_s^* = L^* \cap G_s^*$.
Then $L_s \subseteq G_s$ is a split Levi subgroup of $G_s$ and $L_s^*$ is a dual correspondent.
Since we have $C_{G^*}(s) \subseteq G_s^*$, we also have $C_{L^*}(s) = L^* \cap C_{G^*}(s) \subseteq L^* \cap G_s^* = L_s^*$.
The group $P_s \cap L = L_s (V_s \cap L)$ is a parabolic subgroup of $L$ with Levi complement $L_s$.

As above, the associated Deligne-Lusztig variety $Y^{L}_{L_s \subseteq P_s \cap L}$ has non-vanishing 
cohomology only in degree $d' = \dim(Y^{L}_{L_s \subseteq P_s \cap L})$ and the 
$\Lambda L^F e_s^{L^F} \otimes (\Lambda L_s^F e_s^{L_s^F})^\circ$-module $H^{d'}_c(Y^{L}_{L_s \subseteq P_s \cap L},\Lambda)$
induces a Morita equivalence between $\Lambda L_s^F e_s^{L_s^F}$ and $\Lambda L^F e_s^{L^F}$.

We want to show that the two Morita equivalences just obtained turn Harish-Chandra induction
from $L_s^F$ to $G_s^F$ into Harish-Chandra induction from $L^F$ to $G^F$.

Let $P = L V$ be a rational parabolic subgroup of $G$ having $L$ as Levi complement.
Then $P_1 = (P \cap P_s)V$ and $P_2 = (P \cap P_s)V_s$ are both parabolic subgroups of $G$ having $L_s$
as Levi complement. 
Their unipotent radicals are given by $V_1 = (V_s \cap L)V$ and $V_2 = (G_s \cap V)V_s$, respectively.

We consider their dual correspondents $V_1^*$ and $V_2^*$ and find that $C_{G^*}(s) \subseteq G_s^*$ implies 
$$ C_{V_1^*}(s) = C_{(V_s^* \cap L^*)V^* \cap G_s^*}(s) = C_{G_s \cap V^*}(s) $$
as well as 
$$ C_{V_2^*}(s) = C_{(G_s^* \cap V^*) V_s^* \cap G_s^*}(s) = C_{G_s \cap V^*}(s).$$
Thus, the assumptions of \cite[Cor. 6.5]{bonnafe2015derived} are satisfied and we conclude 
that Lusztig inductions with respect to $P_1$ and $P_2$ are naturally isomorphic up to shifting (and a Tate twist).
Using the transitivity of Lusztig induction (cf. \cite[Thm. 7.9]{cabanes2004representation} and \cite[3.3]{bonnafe2003categories}),  
we find that the diagram 
\[
 \xymatrix{
  D(\Lambda G_s^F e_s^{G_s^F}\mathrm{-mod}) \ar[rr]^{H^d_c(Y^{G}_{G_s \subseteq P_s}, \Lambda) \otimes - }   && D(\Lambda G^F e_s^{G^F}\mathrm{-mod})   \\ \\  
  D(\Lambda L_s^F e_s^{L_s^F}\mathrm{-mod}) \ar[rr]^{H^{d'}_c(Y^{G}_{G_s \subseteq P_s}, \Lambda) \otimes - } 
                                      \ar[uu]^{R^{G_s}_{L_s}}  && D(\Lambda L^F e_s^{L^F}\mathrm{-mod})  \ar[uu]^{R^{G}_{L}} 
 }
\]
is commutative up to shifting and equivalence.
However, since all the functors are exact, with the vertical functors being Harish-Chandra induction, 
they commute with homology which implies that no shifting is required for the diagram to commute 
and so we actually obtain the following result.
\begin{lemma}\label{la: Morita commutative diagram}
Given the notation and assumptions of this section, the diagram 
\[
 \xymatrix{
  \Lambda G_s^F e_s^{G_s^F}\mathrm{-mod} \ar[rr]^{H^d_c(Y^{G}_{G_s \subseteq P_s}, \Lambda) \otimes - }   && \Lambda G^F e_s^{G^F}\mathrm{-mod}   \\ \\  
  \Lambda L_s^F e_s^{L_s^F}\mathrm{-mod} \ar[rr]^{H^{d'}_c(Y^{G}_{G_s \subseteq P_s}, \Lambda) \otimes - } 
                                      \ar[uu]^{R^{G_s}_{L_s}}  && \Lambda L^F e_s^{L^F}\mathrm{-mod}  \ar[uu]^{R^{G}_{L}} 
 }
\]
is commutative up to natural isomorphism.
\end{lemma}

\section{Imprimitivity for unipotent representations of classical groups}
In this section we are going to use the results from Section 4 of 
\cite{geck1996towards}.
Accordingly, we let $(G,F)$ be such that $G^F = G_n(q)$ is one of the groups 
\begin{enumerate}[(a)]
 \item $GL_n(q)$ (any $q$, $n \geq 0$)
 \item $GU_n(q)$ (any $q$, $n \geq 0$)
 \item $Sp_{n}(q)$ ($q$ a power of $2$, $n \geq 0$ even)
 \item $CSp_n(q)$ ($q$ odd, $n \geq 0$ even)
 \item $SO_n(q)$ ($q$ odd, $n \geq 0$ odd)
\end{enumerate}

The reason for restricting to this list of groups
is the following result which is not known for other groups or even known to be at least partially false, for example for the even dimensional orthogonal groups.
\begin{proposition}\label{prop: summary towards}
 Let $(L_0,X_0)$ be a cuspidal pair of $(G,F)$. 
 If $X_0$ is unipotent, 
 then we have 
 $$N_{G^F}(L_0,X_0) = N_{G^F}(L_0)$$ and $X_0$ is extendible to $N_{G^F}(L_0)$.
 
 Moreover, the algebra $\End_{k G^F}(R^G_{L_0} X_0)$ is an Iwahori-Hecke algebra
 associated with the relative Weyl group $N_{G^F}(L_0)/L_0^F$.
\end{proposition}
\begin{proof}
 This was proven in \cite[4.3 and 4.4]{geck1996towards} for the cases 
 (b)--(e). All but the last statements can be proven 
 for case (a) by analogous arguments. 
 The last statement follows in case (a) from  \cite[19.20]{cabanes2004representation}.
\end{proof}

We can now obtain a converse of 
Corollary (\ref{cor: normalizers equal is sufficient}) for 
the unipotent representations of the classical groups we consider.
\begin{theorem}\label{thm: imprimitive unipotent classical}
 Let $(G,F)$ be as in (a)--(e) and let $(X_0,L_0)$ be a cuspidal pair of $(G,F)$ where $X_0$ is unipotent.
 Then the following statements are equivalent:
 \begin{enumerate}[(i)]
  \item There exists a $k G^F$-module in the Harish-Chandra
 series of $(L_0,X_0)$ which is primitive.
  \item Every $k G^F$-module in the Harish-Chandra
 series of $(L_0,X_0)$ is primitive.
  \item We have $N_{G^F}(L_0) \neq N_{L^F}(L_0)$ for every 
  proper split Levi subgroup $L \subseteq G$ containing $L_0$.
  \item We are in one of the cases (b)--(e) or 
        we are in case (a) and in case (a) we either have $L_0^F \cong GL_{1}(q)^n$ or we have $L_0^F \cong GL_{e \ell^i}(q)^m$ where $e$ is the order of $q$ modulo $\ell$
        and $n = m e \ell^i$.
 \end{enumerate}
\end{theorem}
\begin{proof}
 The algebra $\End_{k G^F}(R^G_{L_0} X_0)$ is an Iwahori-Hecke algebra
 by Proposition (\ref{prop: summary towards}) and 
  the embedding 
 $$R^G_L : \End_{k L^F}(R^L_{L_0} X_0) \to \End_{k G^F}(R^G_{L_0} X_0)$$
 identifies the domain with a parabolic subalgebra of this Iwahori-Hecke algebra for any split Levi subgroup $L_0 \subseteq L \subseteq G$.
 
 It follows from 
  \cite[Thm. 1.1]{schoennenbeck2017induced} that 
 if this parabolic subalgebra is a proper one, then    
  the induced module $\Ind_{R^G_L} M$ is reducible
  for every $\End_{k L^F}(R^L_{L_0} X_0)^\circ$-module $M$. 
  On the other hand, if the above embedding is an isomorphism, then 
  clearly $\Ind_{R^G_L} M$ is simple for every simple $\End_{k L^F}(R^L_{L_0} X_0)^\circ$-module $M$.
  
  In view of Lemma (\ref{la: HC commutative diagram}), this implies 
  the equivalence of (i) and (ii). 
  
  Comparing dimensions we also find that either of these statements
  is equivalent to $N_{L^F}(L_0,X_0) \neq N_{G^F}(L_0,X_0)$
  for all proper split Levi subgroups $L_0 \subseteq L \subseteq G$.
  By Proposition (\ref{prop: summary towards}), this is equivalent to 
  (iii).
  
  In the cases (b)--(e), the structure of the normalizers of Levi subgroups admitting cuspidal unipotent representations 
  has been analyzed in the proof of \cite[Prop. 4.3]{geck1996towards}.
  It is easy to see that (iii) is always satisfied in these cases.
  
  In case (a), $L_0^F$ is conjugate in $G^F$ to a group of the form 
  $$ GL_{1}(q)^{m_{-1}} \times \prod_{i = 0}^r GL_{e \ell^{i}}(q)^{m_i} $$
 with $e$ the order of $q$ modulo $\ell$ and $m_{-1},m_0, \dots, m_r \geq 0$  non-negative integers satisfying  
  $n - m_{-1} = e \sum_{i = 1}^r m_i \ell^i$ (cf. \cite[(7.9)]{geck1994cuspidal}).
  
  The group of rational points of the smallest split Levi subgroup containing $N_{G^F}(L_0)$ can now easily be seen to 
  be conjugate in $G^F$ to a group of the form 
 $$ GL_{m_{-1}}(q) \times \prod_{i = 0}^r GL_{e m_i \ell^{i}}(q).$$
  Thus, in case (a), condition (iii) is equivalent to 
  $L_0^F$ being isomorphic to $GL_{1}(q)^n$ or to $GL_{e \ell^i}(q)^m$ where 
  $n = m e \ell^i$.
 This completes the proof.
\end{proof}

\section{Imprimitivity for \texorpdfstring{$GL_n(q)$}{GLn(q)} 
and \texorpdfstring{$GU_n(q)$}{GUn(q)}} 
 
 In this section we let $G = GL_n(\F)$ and $F$ 
 be either 
 the standard Frobenius morphism $F_q$ 
 defined by $F_q((a_{i,j})) = (a_{i,j}^q)$
 or the twisted Frobenius $F_q'$ defined by 
 $F_q'((a_{i,j})) = (a_{i,j}^q)^{-tr}$ for all $(a_{i,j}) \in GL_n(\F)$.
 
 For these groups, we can actually use our results on the Morita
 equivalence by Bonnaf\'{e} and Rouquier together 
 with Theorem (\ref{thm: imprimitive unipotent classical}) to 
 obtain the converse of Corollary (\ref{cor: normalizers equal is sufficient}) for arbitrary Lusztig series.
 
 In the following, we can and will identify $(G,F)$ with its dual.

 \begin{corollary}\label{cor: GL primitive equivalent to normalizers equal}
 Let $M$ be a simple $k G^F$-module which belongs to the Harish-Chandra series of $(L_0,X_0)$.
  Then $M = R^G_{L} M'$ for some split Levi subgroup $L_0 \subseteq L \subseteq G$ 
  if and only if $N_{G^F}(L_0,X_0) = N_{L^F}(L_0,X_0)$. 
\end{corollary}
\begin{proof}
 If $X_0$ is unipotent, then the claim holds by Theorem (\ref{thm: imprimitive unipotent classical}).
 There exists a semisimple element $s \in ({L_0}^{F})_{\ell'}$ such that $X_0$ is an object of $k L_0^F e_s^{L_0^F}\mathrm{-mod}$.
 The groups $G_s = C_G(s)$, $L_s = C_L(s)$ and $L_{0,s} = C_{L_0}(s)$
 are rational Levi subgroups of $G$, $L$ and $L_0$, respectively. 
 We consider the diagram
  \[
  \xymatrix{
   k G^F e_s^{G^F}\mathrm{-mod}\ar[rr] && k G_s^F e_s^{G_s^F}\mathrm{-mod} \\
   k L^F e_s^{L^F}\mathrm{-mod}\ar[rr]\ar[u]^{R^{G}_{L}} && k L_s^F e_s^{L_s^F}\mathrm{-mod} \ar[u]^{R^{G_s}_{L_s}} \\
   k L_{0}^F e_s^{L_0^F}\mathrm{-mod}\ar[rr]\ar[u]^{R^{L}_{L_0}} && k L_{0,s}^F e_s^{L_{0,s}^F}\mathrm{-mod} \ar[u]^{R^{L_s}_{L_{0,s}}}
  }
  \]
 in which the horizontal arrows stand for the respective Bonnaf\'{e}-Rouquier Morita equivalence. 
 By Lemma (\ref{la: Morita commutative diagram}), this diagram commutes up to natural isomorphism.
 Note that $s$ is central in $G_s$, so we have isomorphisms
 $$k G_{s}^F e_s^{{G_s}^F}  \cong k G_{s}^F e_1^{{G_s}^F}$$
 and
 $$k L_{s}^F e_s^{{L_s}^F}  \cong k L_{s}^F e_1^{{L_{s}}^F}$$
 as well as 
 $$k L_{0,s}^F e_s^{{L_{0,s}}^F}  \cong k L_{0,s}^F e_1^{{L_{0,s}}^F}$$
 induced by a linear character $\lambda_s$ of $G_s^F$. 
 As tensoring with linear characters commutes with
  Harish-Chandra induction, the diagram 
 \[
  \xymatrix{
   k G_s^F e_s^{G_s^F}\mathrm{-mod}\ar[rr]^{\lambda_s \otimes -} && k G_s^F e_1^{G_s^F}\mathrm{-mod} \\
   k L_s^F e_s^{L_s^F}\mathrm{-mod}\ar[rr]^{\lambda_s \otimes -}\ar[u]^{R^{G}_{L}} && k L_s^F e_1^{L_s^F}\mathrm{-mod} \ar[u]^{R^{G_s}_{L_s}} \\
   k L_{0,s}^F e_s^{L_{0,s}^F}\mathrm{-mod}\ar[rr]^{\lambda_s \otimes -}\ar[u]^{R^{L}_{L_0}} && k L_{0,s}^F e_1^{L_{0,s}^F}\mathrm{-mod} \ar[u]^{R^{L_s}_{L_{0,s}}}
  }
  \]
  is commutative.
  
 Combining these diagrams we obtain a unipotent $k G_s^F$-module 
 $M_u$ such that $\lambda_s \otimes M_u$ corresponds 
 to $M$ under the Morita equivalence between 
 $k G^F e_s^{G^F}$ and $k G_{s}^F e_s^{{G_s}^F}$.
 In the same way, we obtain a unipotent cuspidal $k L_{0,s}^F$-module
 $X_u$ with $\lambda_s \otimes X_u$ corresponding to $X_0$
 under the analogous Morita equivalence.
 
 Suppose now that $M = R^G_L M'$ for some $k L^F$-module $M'$.
 We let $M_u'$ be the unipotent $k L_s^F$-module that corresponds 
 to $M'$. We thus have $M_u \cong R^{G_s}_{L_s} M_u'$.
 Since $G_s^F$ is a direct product of general linear groups and general unitary groups 
 we have $N_{G_s^F}(L_{0,s},X_u) = N_{L_s^F}(L_{0,s},X_u)$
 by Theorem (\ref{thm: imprimitive unipotent classical}).
 
 Using 
 $$\End_{k L^F}(R^L_{L_0} X_0) \cong \End_{k L_s^F}(R^{L_s}_{L_{0,s}} X_{u})$$
 and 
 $$\End_{k G^F}(R^G_{L_0} X_0) \cong \End_{k G_s^F}(R^{G_s}_{L_{0,s}} X_{u})$$
 and comparing dimensions, we obtain   
 $N_{G^F}(L_{0},X_0) = N_{L^F}(L_{0},X_0)$ as desired.
\end{proof}

 \section{Acknowledgments}
 The author was supported by the DFG-collaborative research center TRR 195.

\bibliographystyle{abbrv}

\end{document}